\documentclass[12pt]{amsart}
\usepackage{amsmath,amssymb,amsfonts}
\usepackage[all,cmtip]{xy}
\usepackage{hyperref}
\usepackage{a4wide}

\newtheorem{thm}{Theorem}[section]
\newtheorem{lem}[thm]{Lemma}

\newtheorem{prop}[thm]{Proposition}

\newtheorem{question}[thm]{Question}

\theoremstyle{definition}
\newtheorem{rmk}[thm]{Remark}

\newcommand{\cd}{{\rm cd}}

\newcommand{\im}{{\rm im}}

\newcommand{\Char}{{\rm char}}

\newcommand{\Gal}{{\rm Gal}}

\newcommand{\sE}{{\mathcal E}}

\newcommand{\F}{{\mathbb F}}
\newcommand{\G}{{\mathbb G}}

\newcommand{\Q}{{\mathbb Q}}

\begin{document}
\title[]{Embedding problems with local conditions and the admissibility of finite groups}
\author{ Nguy\^e\~n Duy T\^an }
 \address{ Universit\"at Duisburg-Essen, FB6, Mathematik, 45117 Essen, Germany \\
 and Institute of Mathematics, 18 Hoang Quoc Viet, 10307, Hanoi-Vietnam.}
\email{duy-tan.nguyen@uni-due.de}
\thanks{
Partially supported by NAFOSTED, the SFB/TR45 and the ERC/Advanced Grant 226257.}

\date{} 

\begin{abstract} 
Let $k$ be a field of characteristic $p>0$, which has infinitely many discrete valuations. We show that every finite embedding problem for $\Gal(k)$  with finitely many prescribed local conditions, whose kernel is a $p$-group, is properly solvable. We then apply this result in studying the admissibility of finite groups over global fields of positive characteristic. We also give another proof for a result of Sonn.

\end{abstract}

\maketitle

\section{Introduction}
By a celebrated theorem of Shafarevich every finite solvable group can be realized as the Galois group of a Galois extension of any given finite algebraic number field $k$. Shafarevich's proof of this theorem however is long and difficult. Seaking for a shorter and more conceptual proof of the theorem leads Neukirch \cite{Ne1, Ne2} naturally to study the embedding problems with local conditions. The works of Neukirch plays also an important role in studying the admissibility of finite groups, a notion due to Schacher \cite{Sch}, see e.g., \cite{N,So2}. See also \cite{HHK,NP} for recent works using patching method to study the admissibility of finite groups.

Let $k$ be a global fields (i.e., a finite extension of $\Q$ or of $\F_p(t)$). For every finite embedding problem $\sE$ for $\Gal(k)$, we can associate a local embedding problem $\sE_v$ for $\Gal(k_v)$, for each prime $v$ of $k$. Each global solution for $\sE$ then gives rise naturally local solutions for $\sE_v$. Now given a collection of local solutions, one can ask whether there is a global one which induces these given local solutions. In the case of number fields or more general in the case that the order of the kernel of our finite embedding problems are prime to the characteristic of $k$, there are many results which give an affirmative answer to that question, see for example \cite{Ne1, Ne2, Ste}.

However, to the best of our knowledge, there are not as much results in the case that the kernel of the embedding problems are divisible by the characteristic of $k$. In this direction, there is a result of Sonn which says that for any embedding problem $1\to A\to E\to \Gal(K/k)\to 1$ for a global field $k$ of characteristic $p>0$ where $E$ is a $p$-group, every given finite set of local solutions is induced from a global one (see \cite[Theorem 1]{So1}).

In this paper, we consider a field $k$ of characteristic $p>0$ equipped with an infinitely many number of discrete valuations and we consider the finite embedding problems for $k$ whose kernels are $p$-groups. We show that for such embedding problems with finitely many prescribed local conditions, it is always properly solvable. See Theorem \ref{thm:solvable}. This result contains the above mentioned result of Sonn as a special case.

We then use Theorem \ref{thm:solvable} to study the admissibility of finite groups over global function fields.
Let $k$ be a field. Following Schacher \cite{Sch}, a finite group $G$ is called {\it $k$-admissible} if there exists a finite Galois extension $L/k$ with Galois group isomorphic to $G$ such that $L$ is a maximal commutative subfield of some finite-dimensional central division algebra over $k$.
 We obtain a reduction theorem for the admissibility of finite groups.  

\begin{thm}
\label{thm:adm}
Let $k$ be a global field of characteristic $p>0$. Let $\Gamma$ be a finite group, $P$ a normal $p$-subgroup of $\Gamma$. If the quotient group $\Gamma/P$ is $k$-admissible then $\Gamma$ is $k$-admissible.
\end{thm}

We note that Stern, in \cite[Theorem 1.3]{Ste}, proves the above result under a stronger assumption that $P$ is a normal $p$-\emph{Sylow} subgroup of $\Gamma$.

In the last section, as an application of Theorem \ref{thm:solvable}, we present another (shorter) proof of a result of Sonn, see Theorem \ref{thm:Sonn}.
\\
\\
\noindent {\bf Acknowledgements:} We would like to give our sincere thanks to H\'el\`ene Esnault for her support and constant encouragement. We would like to thank Danny Neftin for his many interesting comments and remarks, which improve substantially the paper. Thanks are also due to Lior Bary-Soroker for his helps.


\section{$p$-embedding problems with local conditions}
An {\it embedding problem} $\sE$ for a profinite group $\Pi$ is a diagram 
\[
\sE:=
\xymatrix
{
{} & \Pi \ar[d]^{\alpha}\\
\Gamma \ar[r]^f & G
}
\] 
which consists of a pair of profinite groups $\Gamma$ and $G$ and epimorphisms $\alpha:\Pi\to G$, $f:\Gamma\to G$. (All homomorphsims of profinite groups considered in this paper are assumed to be continuous)

A {\it weak solution} of $\sE$ is  a homomorphism $\beta:\Pi\to \Gamma$ such that $f\beta=\alpha$. If such a $\beta$ is surjective, then it is called a {\it proper solution}. We will call $\sE$  {\it weakly} (respectively {\it properly}) {\it solvable} if it has a weak (respectively proper) solution. We call $\sE$ a {\it finite} embedding problem if the group $\Gamma$ is finite. The {\it kernel} of $\sE$ is defined to be $N:=\ker(f)$. We call $\sE$ a {\it p-embedding problem} if $N$ is a  $p$-group. 

Let $\phi_1:\Pi_1\to \Pi$ be a homomorphism of profinite groups. Then the embedding problem $\sE$ induces an embedding problem, which we call the {\it pullback} of $\sE$ to $\Pi_1$ via $\phi_1$,
\[
\phi^*(\sE):=
\xymatrix
{
{} & \Pi_1 \ar[d]^{\alpha\phi_1}\\
\Gamma_1 \ar[r]^{f_1} & G_1,
}
\] 
where $G_1=\alpha\phi_1(\Pi_1)\subset G$, $\Gamma_1=f^{-1}(G_1)\subset \Gamma$, and $f_1=f\mid_{\Gamma_1}$. Note that $\sE$ and $\phi_1^*(\sE)$ have the same kernel.

If $\beta:\Pi\to \Gamma$ is a weak solution to $\sE$, then there is an induced weak solution to $\phi_1^*(\sE)$, namely the {\it pullback} $\phi_1^*(\beta):=\beta\phi_1:\Pi_1\to \Gamma_1$.

Suppose that $\phi=\{\phi_j\}_{j\in J}$ is a family of homomorphisms $\phi_j:\Pi_j\to \Pi$ of profinite groups. 
We will say that $\sE$ is {\it weakly} (respectively {\it properly}) $\phi$-{\it solvable} if for every collection $\{\beta_j\}_{j\in J}$ of weak solutions to the pullback embeddings problems $\phi_j^*(\sE)$, there is a weak (respectively proper) solutions $\beta$ to $\sE$ and elements $n_j\in N=\ker(\sE)$ such that $\phi_j^*(\beta)={\rm inn}(n_j)\circ \beta_j$ for all $j\in J$. (Here ${\rm inn}(n_j)\in {\rm Aut}(\Gamma)$ denotes left conjugation by $n_j$).

Let $\phi=\{\phi_j \}_{j\in J}$ be a family of homomorphisms $\phi_j:\Pi_j \to \Pi$ of profinite groups. We call $\phi$ is {\it strongly $p$-dominating} if the induced map 
$$\phi^{*}: H^1(\Pi,P)\to \prod_{j\in J} H^1(\Pi_j,P)$$
is surjective with infinite kernel for every non-trivial finite elementary abelian $p$-group $P$ on which $\Pi$ acts continuously. 


We will use the following theorem due to Harbater (see \cite[Theorem 2.3]{Ha1} and \cite[Theorem 1]{Ha2}).
\begin{thm}[Harbater]
\label{thm:Ha}
 Let $p$ be a prime number and let $\Pi$ be a profinite group. Consider the following four conditions (i)-(iv):
\begin{itemize}
\item[(i)] Every finite $p$-embedding problem for $\Pi$ is weakly solvable (i.e. $\cd_p(\Pi)\leq 1$).
\item[(ii)] Every finite $p$-embedding problem for $\Pi$ is weakly $\phi$-solvable, for every $p$-dominating
family of homomorphisms $\phi = \{\phi_j : \Pi_j \to \Pi\}_{j\in J}$.
\item[(iii)] Every finite $p$-embedding problem for $\Pi$ is properly solvable.
\item[(iv)] Every finite p-embedding problem for $\Pi$ is properly $\phi$-solvable, for every strongly
$p$-dominating family of homomorphisms $\phi = \{\phi_j : \Pi_j \to \Pi\}_{j\in J}$.
\end{itemize}
Then we have the implications (iii) $\Rightarrow$ (i) $\Leftrightarrow$ (ii) $\Rightarrow$ (iv). 
\end{thm}

In this note, by a {\it $k$-group}, where $k$ is a field, we mean  an algebraic affine group scheme which is smooth (\cite{Wa}). This notion is equivalent to the notion of a linear algebraic group defined over $k$ in the sense of \cite{Bo}.

First we need the following lemma. Recall that a polynomial $f(T)\in k[T]$ in one variable $T$, with coefficients in a field $k$ of characteristic $p$, is called a {\it $p$-polynomial} if $f(T)$ is of the form $f(T)=\sum_{i=0}^m b_i T^{p^i}$, $b_i\in k$.

\begin{lem}
\label{lem:additivepol}
Let $k$ be an infinite field of characteristic $p>0$. Let $P$ be a nontrivial finite commutative $k$-group which is  annihilated by $p$. Then $P$ is $k$-isomorphic to a $k$-subgroup of the additive group $\G_a$, of the form $\{x \mid f(x)=0\},$
where $f(T)=T+b_1T^p+\cdots+b_mT^{p^m}$ is a $p$-polynomial with coefficients in $k$, $m\geq 1$ and $b_m\not=0$. 
\end{lem}
\begin{proof} This is well known, see e.g.\ \cite[Proposition B.1.13]{CGP} or \ \cite[Chapter V, Proposition 4.1 and Subsection 6.1]{Oe}.
\end{proof}

Let $k$ be a field of characteristic $p>0$, and $v$ a valuation of $k$. Let $k_v$ be the completion of $k$ at $v$. We will fix an embedding $\phi_v:\Gal(k_v)\to \Gal(k)$.

\begin{prop}
\label{prop:surj}
Let $k$ be a field of characteristic $p>0$. Let $\Omega$ be the set of non-equivalent valuations of rank 1 of $k$. Let $S$ be a finite subset of $\Omega$. Let $\phi_S$ be the family of homomorphisms $\{\phi_v:\Gal(k_v)\to  \Gal(k)\}_{v\in S}$. 
\begin{enumerate}
 \item The family $\phi_S$ is $p$-dominating.
 \item If there is a valuation $w\in \Omega\setminus S$ whose value group is non-$p$-divisible then the family $\phi_S$ is strongly $p$-dominating.
\end{enumerate}
\end{prop}
\begin{proof}
 Let $P$ be a non-trivial elementary $p$-group on which $\Gal(k)$ acts (continuously). For the part (1), we need to show that
\[
 \phi_S^*: H^1(\Gal(k),P)\to \prod_{v\in S} H^1(\Gal(k_v),P)
\]
is surjective and for the part (2), we need to show further that $\phi_S^*$ has an infinite kernel.

Consider $P$ as a finite $k$-group. Then $P$ is commutative and annihilated by $p$. Hence by Lemma~\ref{lem:additivepol}, $P$ is $k$-isomorphic to a subgroup of $\G_a$ defined as the kernel of $f: \G_a\to \G_a$,  where $f(T)=T+\cdots+b_mT^{p^m}$ is a $p$-polynomial in one variable with coefficients in $k$ with $m\geq 1$ and  $b_m\not=0$. We have the following exact sequence of $k$-groups
$$
0\to P\to \G_a \stackrel{f}{\to}\G_a\to 0.  
$$
From this exact sequence we get the following exact sequence of Galois cohomology groups
$$
H^0(L,\G_a)\stackrel{f}{\to} H^0(L,\G_a)\to H^1(L,P)\to H^1(L,\G_a),
$$
for any field extension $L\supset k$.

By Hilbert 90 $H^1 (L,\G_a)=0$ (see e.g.\ \cite[Chapter II,  Proposition 1]{Se}), hence
$$
H^1(\Gal(L),P)=H^1(L,P)\simeq H^0(L,\G_a)/\im(f)=L/f(L),
$$ 
for any field extension $L\supset k$. 

In particular, $H^1(\Gal(k),P)=k/f(k)$, $H^1(\Gal(k_v),P)=k_v/f(k_v)$ and the map $\phi_S^*$ becomes the canonical map
\[
 \varphi_S:k/f(k) \to \prod_{v\in S} k_v/f(k_v).
\]
(1) 
By the argument as in \cite[Proof of Lemma 2]{TT1}, $\varphi_S$ is surjective. Since the argument is short,  we will present it here for the convenience of the reader.
 
Since $f$ is a separable morphism, we may apply the implicit function theorem and deduce that the subgroup $f(k_v)\subset k_v$ is open. 

Considering $k$ as embedded into $\prod_{v\in S} k_v$, then $k$ is dense in the product $\prod_{v\in S} k_v$ (the weak approximation theorem) and  $\prod_{v\in S} f(k_v)$  is open there. Therefore, $k+\prod_{v\in S} f(k_v)=\prod_{v\in S} k_v$, and $\phi^*_S$ is surjective. 
\\
\\
{\noindent (2)} Set $S^\prime=S\cup\{w\}$. Then, as above, the map 
\[ 
\phi^*_{S^\prime}: H^1(\Gal(k),P) \to \prod_{v\in S} H^1(\Gal(k_v),P) \times H^1(\Gal(k_w),P)
\]
 is surjective. 

Furthermore, $H^1(\Gal(k_w),P)$ is infinite (see \cite[Proof of Theorem 1.1]{BT}). 
 Thus the kernel of $\phi_S^{*}$ is infinite.
\end{proof}
 We have the following main theorem of the paper.
\begin{thm} 
\label{thm:solvable}
Let $k$ be a field of characteristic $p>0$, which has infinitely many discrete valuations. Let $S$ be a finite set of non-equivalent discrete valuations of $k$ and let $\phi_S$ be the family of homomorphisms $\{\phi_v:\Gal(k_v)\to  \Gal(k)\}_{v\in S}$. Then every finite $p$-embedding problem  for $\Gal(k)$ is properly $\phi_S$-solvable.
\end{thm}

\begin{proof}
Since $\Char (k)=p $, we have $\cd_p(\Gal(k))\leq 1$ (see \cite[Chapter II, Proposition 3]{Se}). By Proposition \ref{prop:surj} part (2), the family $\phi_S$ is strongly $p$-dominating. Therefore, by Theorem \ref{thm:Ha}, every finite $p$-embedding problem for $\Gal(k)$ is properly $\phi_S$-solvable.
\end{proof}

\section{Admissibility of finite groups over global function fields}
In this section we will use Theorem \ref{thm:solvable} to study the admissibility of finite groups over global function fields (Theorem \ref{thm:adm})

To prove Theorem \ref{thm:adm}, we need the following criterion of admissibility of finite groups over global fields, which is due to Schacher (see \cite[Propositions 1, 5 and 6]{Sch}).

\begin{thm}[Schacher]
\label{thm:Sch}
Let $k$ be a global field, $G$ a finite group. Then $G$ is $k$-admissible if and only if there exists a finite Galois extension $L/k$ with Galois group isomorphic to $G$ such that for every prime $l$ dividing the order of $G$, $\Gal(L_v/l_v)$ contains an $l$-Sylow subgroup of $G$ for at least two different primes $v$ of $k$, where $k_v$, $L_v$ are the completions of $k$ and $L$ at $v$, respectively. 
\end{thm}

We are now ready to prove Theorem \ref{thm:adm}.

\begin{proof}[Proof of Theorem \ref{thm:adm}]
 Set $G=\Gamma/P$. By Theorem \ref{thm:Sch}, there exists a finite Galois extension $L/k$ with Galois group $\Gal(L/k)\simeq G$ such that for every prime $l$ dividing the order of $G$, $\Gal(L_v/k_v)$ contains an  $l$-Sylow subgroup of $G$ for at least two different primes $v$ of $k$, where $k_v$, $L_v$ are the completion of $k$ and $L$ at $v$, respectively. We have the following diagrams
\[ \sE:=
\xymatrix
{
{} & {} & {} & \Gal(k) \ar[d]^{\alpha}\\
1 \ar[r] & P \ar[r] &\Gamma \ar[r]^-f & {G=\Gal(L/k)}\ar[r] & 1
}
\]
and 

\[ \sE_v:=
\xymatrix
{
{} & {} & {} & \Gal(k_v) \ar[d]^{\alpha_v}\\
1 \ar[r] & P \ar[r] &\Gamma_v \ar[r]^-f & G_v=\Gal(L_v/k_v)\ar[r] & 1
}
\]
for each discrete valuation $v$ of $k$, where $\Gamma_v=f^{-1}(G_v)$.

Let $l$ be any prime number which divides the order of $\Gamma$. There are two discrete valuations $v_1(l), v_2(l)$ of $k$ such that for each $v$ in $\{v_1(l),v_2(l)\}$, $\Gal(L_v/k_v)$ contains a $l$-Sylow group of $G$. Set $S$ be the union of all $\{v_1(l),v_2(l)\}$, where $l$ runs over the finite set of prime divisors of the order of $\Gamma$.  For each $v$ in $S$, the embedding problem $\sE_v$ has a weak solution $\beta_v:\Gal(k_v)\to \Gamma_v$ since $\cd_p(\Gal(k_v))\leq 1$. (In fact, $\sE_v$ has a proper solution, see \cite[Theorem 1.1]{BT}).

  By Theorem \ref{thm:solvable}, there exists a proper solution $\beta:\Gal(k)\to \Gamma$ of $\sE$ such that its induced local solution $\beta|_{\Gal(k_v)}:\Gal(k_v)\to\Gamma_v$ is equal to $\beta_v$ up to an inner automorphism by an element in $P$, for each $v$ in $S$. 

Let $K/k$ be the Galois extention corresponding to the solution $\beta$. Then $\Gal(K/k)=\Gamma$ and $\Gal(K_v/k_v)=\Gamma_v$. Furthermore, for each $v$ in $S$, $\Gamma_v$ contains an $l$-Sylow subgroup of $\Gamma$ since $G_v$ contains an $l$-Sylow subgroup $G$. By Theorem \ref{thm:Sch}, $\Gamma$ is $k$-admissible.
\end{proof}

Inspired by recent works of \cite{HHK, NP} on the admissibility of finite groups, we would like to raise the following question.

\begin{question}
 Does Theorem \ref{thm:adm} still hold true for other fields of characteristic $p>0$ besides global function fields, e.g., 
\begin{enumerate}
 \item A field which is finitely generated field extension of $K$ of transcendence degree one, where $K$ is complete with respect to a discrete valuation and whose residue field is algebraically closed,
\item The fraction field of a complete local domain of dimension 2, with a separably closed residue field?
\end{enumerate}
\end{question}

\section{A result of Sonn}

Let $k$ be a field, $G$ a finite group. In \cite{FS}, the authors say that the pair $(k,A)$ has Property A if every division algebra with center $k$ and index equal to the order of $G$ is a crossed product for $G$. They show that if $k$ if $k$ is a global field of characteristic not dividing the order of $G$ then $(k,G)$ has Property A if and only if $G$ is the direct product of two cyclic groups $C_e$ and $C_f$ of order $e, f$, respectively, and $k$ contains the $e$-th roots of unity; if $k$ is a global field of finite characteristic $p$ dividing the order of $G$, then a necessary condition that $(k,G)$ has Property A is that $G$ have a normal $p$-Sylow subgroup with quotient group $C_e\times C_f$, $e$ and $f$ prime to $p$, and $k$ contains the $e$-th roots of unity. In \cite{So1}, Sonn show that the above condition is also sufficient by proving the following theorem.

\begin{thm}[{\cite[Theorem 3]{So1}}]
\label{thm:Sonn}
Let $k$ be a global field of characteristic $p>0$, $e$ and $f$ positive integers prime to $p$, and assume that $k$ contains the $e$-th roots of unity. Let $\Gamma$ be a finite group whose $p$-Sylow group $P$ is normal, with factor group $G$ isomorphic to the direct product of two cyclic groups of order $e$ and $f$, respectively. Let $S$ be any finite set of primes of $k$. Then there exists a Galois extension $K$ of $k$ with $\Gal(K/k)\simeq \Gamma$ such that for each $v\in S$, $\Gal(K_v/k_v)\simeq \Gamma$, where $L_v$, $k_v$ denote the completion of $L$, $k$, respectively, at $v$.
\end{thm}
Using Theorem \ref{thm:solvable}, we can also give another proof of the above result of Sonn.
\begin{proof}
 For each $v\in S$, let $L_v$ be the (tamely ramified) extension of $k_v$ generated by the unramified extension of $k_v$ of degree $f$ and by the $e$-th root of a prime element of $k_v$. Since $k_v$ contains the $e$th roots of unity, this is an abelian extension with Galois group isomorphic to $G$. By Gruenwald's theorem (\cite[Chapter 10, Theorem 5]{AT}), there is a Galois extension $L/k$ with $\Gal(L/k)\simeq G$ whose completions coincide with $L_v$. That means the following embedding problem
\[ \sE:=
\xymatrix
{
{} & {} & {} & \Gal(k) \ar[d]^{\alpha}\\
1 \ar[r] & P \ar[r] &\Gamma \ar[r]^-f & {G=\Gal(L/k)}\ar[r] & 1
}
\]
induces the following local one
\[ \sE_v:=
\xymatrix
{
{} & {} & {} & \Gal(k_v) \ar[d]^{\alpha_v}\\
1 \ar[r] & P \ar[r] &\Gamma \ar[r]^-f & G=\Gal(L_v/k_v)\ar[r] & 1
}
\]
For each $v$ in $S$, the embedding problem $\sE_v$ has a proper solution $\beta_v:\Gal(k_v)\to \Gamma$ (see \cite[Theorem 1.1]{BT}).   By Theorem \ref{thm:solvable}, there exists a proper solution $\beta:\Gal(k)\to \Gamma$ of $\sE$ such that its induced local solution $\beta|_{\Gal(k_v)}:\Gal(k_v)\to\Gamma$ is equal to $\beta_v$ up to an inner automorphism by an element in $P$, for each $v$ in $S$. 

Let $K/k$ be the Galois extension corresponding to the solution $\beta$. Then we have $\Gal(K/k)=\Gamma$ and $\Gal(K_v/k_v)=\Gamma$ as required. 
\end{proof}

\end{document}